\documentclass[a4paper,12pt]{article}

\usepackage[left=2cm,right=2cm, top=2cm,bottom=3cm,bindingoffset=0cm]{geometry}

\usepackage{verbatim}
\usepackage{amsmath}
\usepackage{amsthm}
\usepackage{amssymb}
\usepackage{delarray}
\usepackage{cite}
\usepackage{hyperref}
\usepackage{mathrsfs}
\usepackage{tikz}
\usetikzlibrary{patterns}
\usepackage{caption}
\DeclareCaptionLabelSeparator{dot}{. }
\newcommand{\al}{\alpha}

\newcommand{\de}{\delta}
\newcommand{\la}{\lambda}

\newcommand{\eps}{\varepsilon}
\newcommand{\vv}{\varphi}
\newcommand{\iy}{\infty}

\theoremstyle{plain}

\numberwithin{equation}{section}

\newtheorem{thm}{Theorem}[section]
\newtheorem{lem}[thm]{Lemma}

\newtheorem{cor}[thm]{Corollary}

\theoremstyle{definition}

\newtheorem{example}[thm]{Example}
\newtheorem{alg}[thm]{Algorithm}
\newtheorem{ip}[thm]{Inverse Problem}

\theoremstyle{remark}

 \allowdisplaybreaks

\begin{document}

\begin{center}
{\Large\bf Finite-difference approximation of the inverse\\[0.2cm]
Sturm-Liouville problem with frozen argument}
\\[0.2cm]
{\bf Natalia P. Bondarenko} \\[0.2cm]
\end{center}

\vspace{0.5cm}

{\bf Abstract.} This paper deals with the discrete system being the finite-difference approximation of the Sturm-Liouville problem with frozen argument. The inverse problem theory is developed for this discrete system. We describe the two principal cases: degenerate and non-degenerate.  For these two cases, appropriate inverse problems statements are provided, uniqueness theorems are proved,  and reconstruction algorithms are obtained. Moreover, the relationship between the eigenvalues of the continuous problem and its finite-difference approximation is investigated. We obtain the ``correction terms'' for approximation of the discrete problem eigenvalues by using the eigenvalues of the continuous problem. Relying on these results, we develop a numerical algorithm for recovering the potential of the Sturm-Liouville operator with frozen argument from a finite set of eigenvalues. The effectiveness of this algorithm is illustrated by numerical examples.

\medskip


\medskip


\vspace{1cm}

\section{Introduction} \label{sec:intr}

This paper is concerned with the finite-difference approximation
\begin{gather} \label{eqv1}
-\frac{y_{j + 1} - 2 y_j + y_{j - 1}}{h^2} + q_j y_m = \la y_j, \quad j = \overline{1, l},
\\ \label{bc1}
y_0 = y_{l + 1} = 0
\end{gather}
of the eigenvalue problem for the functional-differential Sturm-Liouville equation
\begin{gather} \label{eqv2}
    -y''(x) + q(x) y(a) = \la y(x), \quad x \in (0, \pi), \\ \label{bc2}
    y(0) = y(\pi) = 0.
\end{gather}
Here $\la$ is the spectral parameter, $a \in (0, \pi)$ is the so-called \textit{frozen argument}, $q \in L_2(0, \pi)$ is a complex-valued function called \textit{the potential}, $l \in \mathbb N$, $h = \frac{\pi}{l + 1}$, $y_j = y(x_j)$, $q_j = q(x_j)$, $x_j = j h$, $j = \overline{0, l + 1}$, the index $m \in \{ 1, \ldots, l\}$ is such that $x_m = a$.

Equation~\eqref{eqv2} belongs to the class of the so-called \textit{loaded equations} \cite{Nakh12, Nakh82}, which contain the values of the unknown function at some fixed points. Such equations have been applied for studying groundwater dynamics \cite{Nakh82, NB77}, heating processes \cite{Isk71, DKN76}, feedback-like phenomena in vibrations of a wire affected by a magnetic field \cite{Krall75, BH21}. It is also worth mentioning that the operators induced by equation~\eqref{eqv2} are adjoint to the certain differential operators with nonlocal integral boundary conditions (see \cite{Lom14, Pol21}), arising in the investigation of diffusion processes \cite{Fel54}. Since loaded equations are difficult for analysis, it is natural to develop numerical methods for their solution. In particular, finite-difference approximations of the boundary value problems (but not eigenvalue problems) for ordinary and partial loaded differential equations were constructed and analyzed in \cite{BV00, ABSL08, SL09, AA16}. Exact solutions for a special class of loaded difference equations were obtained in \cite{PP18-1, PP18-2}.

This paper is mostly focused on the theory of inverse spectral problems. Such problems consist in the recovery of operators from their spectral characteristics. The most significant progress in inverse problem theory was achieved for the differential Sturm-Liouville equation (see the monographs \cite{Mar77, Lev84, PT87, FY01}):
\begin{equation} \label{StL}
-y''(x) + q(x) y(x) = \la y(x).
\end{equation}
However, equation~\eqref{eqv2} with frozen argument is nonlocal, therefore the classical methods of inverse problem theory do not work for it.

Inverse problems for the functional-differential equation~\eqref{eqv2} were studied in \cite{AHN07, Niz09, BBV19, BV19, XY19, BK20, HBY20, WZZW21, BH21, TLBCS21}. In particular, it has been shown in \cite{BBV19} that the spectrum of the problem~\eqref{eqv2}-\eqref{bc2} is the countable set of the eigenvalues $\{ \la_n \}_{n = 1}^{\iy}$ (counting with multiplicities) with the asymptotics
\begin{equation} \label{asymptla0}
\la_n = \left( n + \frac{\varkappa_n}{n} \right)^2, \quad n \in \mathbb N, \quad \{ \varkappa_n \} \in l_2.
\end{equation}

The inverse spectral problem for the Sturm-Liouville operator with frozen argument is formulated as follows.

\begin{ip} \label{ip:cont}
Given the eigenvalues $\{ \la_n \}_{n = 1}^{\iy}$, find the potential $q$.
\end{ip}

The uniqueness of the recovery of the potential $q(x)$ from the eigenvalues $\{ \la_n \}_{n = 1}^{\iy}$ depends on $a$. There are two principal cases:

\textbf{Degenerate case} (see \cite{BBV19, BV19, BK20, TLBCS21}). If $\frac{\pi}{a} \in \mathbb Q$, that is, $\frac{\pi}{a} = \frac{j}{k}$, $j, k \in \mathbb N$, $j < k$, then a part of the spectrum degenerates:
\begin{equation} \label{deg}
\la_{nk} = (nk)^2, \quad n \in \mathbb N.
\end{equation}
In this case, the potential $q(x)$ cannot be uniquely recovered from the spectrum $\{ \la_n \}_{n = 1}^{\iy}$. However, in \cite{BK20}, the classes of iso-spectral potentials are described and additional restrictions sufficient for the uniqueness of the inverse problem solution are provided. Furthermore, in \cite{BBV19, BV19, BK20} constructive algorithms for solving the inverse problems together with the necessary and sufficient conditions of their solvability are obtained.

\textbf{Non-degenerate case} (see \cite{WZZW21}). If $\frac{\pi}{a} \not\in \mathbb Q$, then there is no degeneration of form~\eqref{deg} and the solution of Inverse Problem~\ref{ip:cont} is unique in $L_2(0, \pi)$.

In this paper, we develop the inverse problem theory for the discrete system~\eqref{eqv1}-\eqref{bc1}, having a finite set of the eigenvalues $\{ \la_{n,l} \}_{n=1}^l$. We study the following inverse problem.

\begin{ip} \label{ip:disch}
Given the eigenvalues $\{ \la_{n,l} \}_{n = 1}^l$, find $\{ q_j \}_{j = 1}^l$.
\end{ip}

We show that, for the discrete problem, there are also two principal cases, similar to the ones for the continuous problem.

\textbf{Degenerate case.} If $\gcd(m, l + 1) = d > 1$, then a part of the spectrum degenerates:
\begin{equation} \label{degh}
\la_{n,l} = \frac{4 \sin^2\left( \frac{nh}{2}\right)}{h^2}, \quad n = kr, \quad r = \frac{l + 1}{d}, \quad k = \overline{1, d-1}. 
\end{equation}
In this case, solution of Inverse Problem~\ref{ip:disch} is non-unique.

\textbf{Non-degenerate case.} If $\gcd(m, l + 1) = 1$, then the eigenvalues $\{ \la_{n,l} \}_{n = 1}^l$ uniquely specify $\{ q_j \}_{j =1}^l$.

In Section~\ref{sec:disc} of this paper, we discuss additional data for the unique solvability of Inverse Problem~\ref{ip:disch} in the degenerate case, formulate the uniqueness theorems, and develop constructive algorithms for solving the inverse problems for the degenerate and non-degenerate cases. 

In the inverse problem solution, an important role is played by the Chebyshev polynomials of the second kind. The crucial step of the reconstruction of the coefficients $\{ q_j \}$ is finding the coordinates of some characteristic polynomials with respect to the basis of the Chebyshev polynomials. Note that the continuous problem~\eqref{eqv2}-\eqref{bc2} with frozen argument also has a deep relationship with the Chebyshev polynomials (see \cite{TLBCS21}).

In Section~\ref{sec:asympt}, we study the connection between the eigenvalues $\{ \la_n \}_{n = 1}^{\iy}$ of the discrete problem~\eqref{eqv1}-\eqref{bc1} and the eigenvalues $\{ \la_{n,l} \}_{n = 1}^l$ of the continuous problem~\eqref{eqv2}-\eqref{bc2}. For the Sturm-Liouville equation \eqref{StL}, the relationship between the eigenvalues of the continuous problem and its finite-difference approximation has been investigated in \cite{PHA81} and used for development of a numerical algorithm for recovery of the potential $q(x)$ from the eigenvalues in \cite{FKL95}. Note that the methods of~\cite{PHA81, FKL95} cannot be directly applied to the problem~\eqref{eqv1}-\eqref{bc1} because it is non-self-adjoint. For studying the behavior of the eigenvalues $\la_{n,l}$ as $l \to \iy$, we develop three different methods for ``small'', ``middle'', and ``large'' $n$. As a result, we show that $\la_{n,l}$ can be approximated by the continuous problem eigenvalues $\la_n$ with the ``correction terms'': 
$$
\la_{n,l} \approx \la_n - n^2 + \frac{4 \sin^2 \left( \frac{nh}{2}\right)}{h^2}.
$$
The situation is different from the results of \cite{PHA81, FKL95}, since for the standard Sturm-Liouville equation, only a part of the eigenvalues of the corresponding discrete problem have a significantly good approximation by using the ``correction terms''.

In Section~\ref{sec:num}, we use the results of Sections~\ref{sec:disc}-\ref{sec:asympt} for the development of a numerical algorithm for recovering the potential $q(x)$ from a finite set of the eigenvalues $\{ \la_n \}$ of the problem~\eqref{eqv2}-\eqref{bc2}. Although there is a number of numerical methods for solution of the classical inverse Sturm-Liouville problems (see, e.g., \cite{RS92, IY08}), there are only a few results in this direction for nonlocal operators. In particular, a numerical approach to the inverse problem solution was developed in \cite{BB20} for one special class of the second-order integro-differential operators. For functional-differential operators with frozen argument, as far as we know, numerical solution of inverse spectral problems has not been studied yet. In this paper, we make the first steps in this direction. In Section~\ref{sec:num}, we describe a numerical method based on the finite-difference approximation~\eqref{eqv1}-\eqref{bc1} for solving the inverse problem for the Sturm-Liouville operator with frozen argument. Numerical examples confirming the effectiveness of this method are provided.

In Section~\ref{sec:gen}, possible generalizations and applications of our results are discussed.

\section{Solution of the discrete inverse problem} \label{sec:disc}

Let us rewrite the discrete system~\eqref{eqv1}-\eqref{bc1} in the following equivalent form
\begin{gather} \label{eqv3}
    y_{j + 1} + y_{j - 1} - w_j y_m = \mu y_j, \quad j = \overline{1, l}, \\ \label{bc3}
    y_0 = y_{l + 1} = 0,
\end{gather}
where $w_j = h^2 q_j$, $\mu = 2 - h^2 \la$. 

In this section, the inverse problem theory for the eigenvalue problem~\eqref{eqv3}-\eqref{bc3} is developed. We describe the two principal cases: degenerate and non-degenerate. For these two cases, we provide appropriate inverse problems statements, prove uniqueness theorems, and obtain reconstruction algorithms.

Denote by $[P_j(\mu)]_{j = 0}^{l + 1}$ and $[Q_j(\mu)]_{j = 0}^{l + 1}$ the solutions of the system~\eqref{eqv3} satisfying the initial conditions
\begin{equation} \label{icPQ}
P_{m - 1}(\mu) = 1, \quad P_m(\mu) = 0, \quad Q_{m - 1}(\mu) = 0, \quad Q_m(\mu) = 1.
\end{equation}
Obviously, using~\eqref{eqv3} and \eqref{icPQ}, one can readily compute $P_j(\mu)$ and $Q_j(\mu)$ for $j = m-2, m-3, \ldots, 1$ and for $j = m + 1, m + 2, \ldots, l + 1$. 
Below the notation $f(\mu) \sim f_0 \mu^n$ means that $f(\mu)$ is a polynomial of degree $n$ with the leading coefficient $f_0$. One can easily check that
\begin{equation} \label{asymptPQ}
P_0(\mu) \sim \mu^{m-1}, \quad P_{l + 1}(\mu) \sim -\mu^{l - m}, \quad Q_0(\mu) \sim (q_{m-1} - 1) \mu^{m - 2}, \quad Q_{l + 1}(\mu) \sim \mu^{l - m + 1}.
\end{equation}

The eigenvalues of the problem~\eqref{eqv3}-\eqref{bc3} coincide with the zeros of the characteristic function
\begin{equation} \label{char}
D(\mu) = P_0(\mu) Q_{l + 1}(\mu) - P_{l + 1}(\mu) Q_0(\mu).
\end{equation}
It follows from \eqref{asymptPQ} that $D(\mu) \sim \mu^l$.
Hence, the problem \eqref{eqv3}-\eqref{bc3} has exactly $l$ eigenvalues $\{ \mu_n \}_{n = 1}^l$, and
\begin{equation} \label{prod}
    D(\mu) = \prod_{n = 1}^l (\mu - \mu_n).
\end{equation}

Clearly, Inverse Problem~\ref{ip:disch} is equivalent to the following one.

\begin{ip} \label{ip:disc}
Given $\{ \mu_n \}_{n = 1}^l$, find $\{ w_j\}_{j = 1}^l$.
\end{ip}

Proceed with the solution of Inverse Problem~\ref{ip:disc}. Define $\psi_n(\mu)$ by the recurrence relation
$$
\psi_0(\mu) = 0, \quad \psi_1(\mu) = 1, \quad \psi_{n + 1}(\mu) = \mu \psi_n(\mu) - \psi_{n-1}(\mu), \quad n \in \mathbb N.
$$
Clearly, $\psi_n(\mu) = U_{n-1}\left( \frac{\mu}{2}\right)$, where $U_n(\mu)$ are the Chebyshev polynomials of the second kind. Hence 
\begin{equation} \label{sin}
\psi_n(\mu) = \frac{\sin (n \theta)}{\sin \theta}, \quad \mu = 2 \cos \theta,
\end{equation}
$\deg(\psi_n) = n-1$, and the zeros of $\psi_n(\mu)$ are
\begin{equation} \label{rpsi}
2 \cos \left( \frac{\pi k}{n}\right), \quad k = \overline{1, n-1}.
\end{equation}

By induction, one can show that
\begin{gather} \label{Ppsi}
    P_0(\mu) = \psi_m(\mu), \quad P_{l + 1}(\mu) = -\psi_{l - m + 1}(\mu), \\ \label{Qpsi}
    Q_0(\mu) = -\psi_{m-1}(\mu) + \sum_{j = 1}^{m-1} w_j \psi_j(\mu), \quad
    Q_{l + 1}(\mu) = \psi_{l - m + 2}(\mu) + \sum_{j = 1}^{l - m + 1} w_{l + 1 - j} \psi_j(\mu).
\end{gather}
Substituting \eqref{Ppsi}-\eqref{Qpsi} into~\eqref{char}, we obtain
\begin{equation} \label{Dwm}
D(\mu) = \psi_m(\mu) \psi_{l - m + 2}(\mu) + \mu^{l - 1} w_m + O(\mu^{l-m}).
\end{equation}
Hence, given $D(\mu)$, one can find $w_m$ from~\eqref{Dwm}.

\textbf{Non-degenerate case.}
If the numbers $m$ and $(l + 1)$ are relatively prime, that is, $\gcd(m, l + 1) = 1$, then the polynomials $P_0(\mu)$ and $P_{l + 1}(\mu)$ do not have common roots. Substituting their roots
\begin{equation} \label{roots}
\nu_k = 2\cos\left( \frac{\pi k}{m}\right), \quad k = \overline{1, m-1}, \quad \text{and} \quad \theta_k = 2 \cos\left( \frac{\pi k}{l - m + 1}\right), \quad k = \overline{1, l - m},
\end{equation}
into \eqref{char}, we derive
\begin{align} \label{intQ0}
& Q_0(\nu_k) = -P_{l + 1}^{-1}(\nu_k) D(\nu_k), \quad k = \overline{1, m - 1}, \\ \label{intQl}
& Q_{l + 1}(\theta_k) = P_0^{-1}(\theta_k) D(\theta_k), \quad k = \overline{1, l - m}. 
\end{align}
In view of \eqref{Qpsi}, $Q_0(\mu)$ is a polynomial of degree $(m-2)$ and $Q_{l + 1}(\mu)$ is a polynomial of degree $(l - m + 1)$ with the known coefficients at $\mu^{l - m + 1}$ and $\mu^{l - m}$. Therefore, the polynomials $Q_0(\mu)$ and $Q_{l + 1}(\mu)$ can be uniquely constructed by using formulas~\eqref{intQ0}-\eqref{intQl} and interpolation. 

Observe that the Chebyshev polynomials $\{ \psi_j \}_{j = 1}^n$ form a basis in the space $\mathscr P_{n-1}$ of polynomials of degree at most $(n-1)$. Consequently, the numbers $\{ w_j \}_{j = 1}^{m-1}$ can be found from~\eqref{Qpsi} as the coordinates of the polynomial $(Q_0 + \psi_{m - 1}) \in \mathscr P_{m-2}$ with respect to the basis $\{ \psi_j \}_{j = 1}^{m-1}$, and the numbers $\{ w_{l + 1 - j} \}_{j = 1}^{l-m}$, as the coordinates of the polynomial $(Q_{l + 1} - \psi_{l - m + 2} - w_m \psi_{l - m + 1}) \in \mathscr P_{l - m - 1}$ with respect to the basis $\{ \psi_j \}_{j = 1}^{l - m}$. Thus, one can solve Inverse Problem~\ref{ip:disc} by the following algorithm.

\begin{alg} \label{alg:prime}
Suppose that $\gcd(m, l + 1) = 1$ and the eigenvalues $\{ \mu_n \}_{n = 1}^l$ are given. We have to find $\{ w_j \}_{j = 1}^l$.

\begin{enumerate}
    \item Construct $D(\mu)$ by \eqref{prod}.
    \item Determine $w_m$ by using~\eqref{Dwm}.
    \item Find $P_0(\mu)$ and $P_{l + 1}(\mu)$ by~\eqref{Ppsi}.
    \item Calculate $\{ \nu_k \}_{k = 1}^{m-1}$, $\{ \theta_k \}_{k = 1}^{l - m}$ by \eqref{roots} and $\{ Q_0(\nu_k) \}_{k = 1}^{m-1}$, $\{ Q_{l + 1}(\theta_k) \}_{k = 1}^{l-m}$ by \eqref{intQ0}-\eqref{intQl}.
    \item Interpolate the polynomial $Q_0(\mu)$ by its values $\{ Q_0(\nu_k) \}_{k = 1}^{m-1}$.
    \item Interpolate the polynomial $(Q_{l + 1}(\mu) - \mu^{l - m + 1} - w_m \mu^{l - m}) \in \mathscr P_{l-m-1}$ by its values $\{ Q_{l + 1}(\theta_k) - \theta_k^{l - m + 1} - w_m \theta_k^{l - m} \}_{k = 1}^{l-m}$, then find $Q_{l + 1}(\mu)$.
    \item Find the coordinates $\{ w_j \}_{j = 1}^{m-1}$ of the polynomial $(Q_0 + \psi_{m - 1})$ with respect to the basis $\{ \psi_j \}_{j = 1}^{m-1}$.
    \item Find the coordinates $\{ w_{l + 1 - j} \}_{j = 1}^{l-m}$ of the polynomial $(Q_{l + 1} - \psi_{l - m + 2} - w_m \psi_{l - m + 1}) \in \mathscr P_{l - m - 1}$ with respect to the basis $\{ \psi_j \}_{j = 1}^{l - m}$.
\end{enumerate}
Thus, all the elements of $\{ w_j \}_{j = 1}^l$ are found.
\end{alg}

The uniqueness of construction at each step of
Algorithm~\ref{alg:prime} implies the following uniqueness theorem for solution of Inverse Problem~\ref{ip:disc} in the case $\gcd(m, l + 1) = 1$. Along with the problem \eqref{eqv3}-\eqref{bc3} with the coefficients $\{ w_j \}_{j = 1}^l$ and the eigenvalues $\{ \mu_n \}_{n = 1}^l$, we consider another problem of the same form having different coefficients $\{ \tilde w_j \}_{j = 1}^l$ and the eigenvalues $\{ \tilde \mu_j \}_{j = 1}^l$. The parameters $l$ and $m$ are supposed to be the same for these two problems.

\begin{thm}
Suppose that $\gcd(m, l + 1) = 1$ and $\mu_n = \tilde \mu_n$, $n = \overline{1, l}$. Then $w_j = \tilde w_j$, $j = \overline{1, l}$.
\end{thm}

Below we show that the condition $\gcd(m, l + 1) = 1$ is not only sufficient but also necessary for the uniqueness of solution of Inverse Problem~\ref{ip:disc}.

\textbf{Degenerate case.}
Consider the case $\gcd(m, l + 1) = d > 1$. Then the polynomials $P_0(\mu)$ and $P_{l + 1}(\mu)$ have the common roots $\xi_k = 2 \cos \frac{\pi k}{d}$, $k = \overline{1, d-1}$, being the roots of $\psi_d(\mu)$. It follows from~\eqref{char} that $\{ \xi_k \}_{k = 1}^{d-1}$ are also roots of $D(\mu)$ and eigenvalues of the problem~\eqref{eqv3}-\eqref{bc3}: $\mu_{kr} = \xi_k$, $k = \overline{1, d-1}$, $r = \frac{l + 1}{d}$. Thus, relation~\eqref{char} can be represented in the form
\begin{equation} \label{char2}
D(\mu) = \psi_d(\mu) (\vv_{m - d}(\mu) Q_{l + 1}(\mu) + \vv_{l - m + 1 - d}(\mu) Q_0(\mu)),
\end{equation}
where 
\begin{equation} \label{defvv}
\vv_{m-d}(\mu) := \frac{\psi_{m}(\mu)}{\psi_d(\mu)} \quad \text{and} \quad \vv_{l - m + 1 - d}(\mu) := \frac{\psi_{l - m + 1 - d}(\mu)}{\psi_d(\mu)}
\end{equation}
are relatively prime polynomials of degrees $(m - d)$ and $(l - m + 1 - d)$, respectively. Clearly, the degenerate part of the spectrum $\{ \mu_{kr} \}_{k = 1}^{d-1}$ carries no information about $\{ w_j \}_{j = 1}^l$. It is also clear that the $(l - d + 1)$ remaining eigenvalues $\{ \mu_n \}_{n = 1}^l \setminus \{ \mu_{kr} \}_{k = 1}^{d-1}$ are insufficient for the unique reconstruction of the $l$ unknown values $\{ w_j \}_{j = 1}^l$. Therefore, instead of Inverse Problem~\ref{ip:disc}, we consider the following two partial inverse problems.

\begin{ip} \label{ip:disc1}
Suppose that the numbers $\{ w_j \}_{j = m - d + 1}^{m - 1}$ are known a priori. Given the eigenvalues $\{ \mu_n \}_{n = 1}^l \setminus \{ \mu_{kr} \}_{k = 1}^{d-1}$, find $\{ w_j \}_{j = 1}^l \backslash \{ w_j \}_{j = m - d + 1}^{m - 1}$.
\end{ip}

\begin{ip} \label{ip:disc2}
Suppose that the numbers $\{ w_j \}_{j = m + 1}^{m + d}$ are known a priori. Given the eigenvalues $\{ \mu_n \}_{n = 1}^l \setminus \{ \mu_{kr} \}_{k = 1}^{d-1}$, find $\{ w_j \}_{j = 1}^l \backslash \{ w_j \}_{j = m + 1}^{m + d}$. 
\end{ip}

Suppose that the conditions of Inverse Problem~\ref{ip:disc1} are fulfilled.
Denote by $\{ \nu_k \}_{k = 1}^{m - d}$ the roots of the polynomial $\vv_{m-d}(\mu)$. These roots can be easily found by using~\eqref{rpsi} and~\eqref{defvv}. Using~\eqref{char} and~\eqref{char2}, we obtain
$$
Q_0(\nu_k) = -P^{-1}_{l + 1}(\nu_k) D(\nu_k), \quad k = \overline{1, m - d}.
$$
Define the polynomial
\begin{equation} \label{Q0b}
Q_0^{\bullet}(\mu) = Q_0(\mu) + \psi_{m-1}(\mu) - \sum_{j = m-d}^{m-1} w_j \psi_j(\mu).
\end{equation}
Clearly, $\deg(Q_0^{\bullet}) \le (m - d - 1)$ and
\begin{equation} \label{Q0bul}
Q_0^{\bullet}(\nu_k) = -P^{-1}_{l + 1}(\nu_k) D(\nu_k) + \psi_{m-1}(\nu_k) - \sum_{j = m-d}^{m-1} w_j \psi_j(\nu_k), \quad k = \overline{1, m - d}.
\end{equation}
Therefore, one can use interpolation to find $Q_0^{\bullet}(\mu)$ by its values \eqref{Q0bul}, and after that find $Q_0(\mu)$ from \eqref{Q0b}. Further, using~\eqref{char}, we find 
\begin{equation} \label{findQl}
Q_{l + 1}(\mu) = P_0^{-1}(\mu) (D(\mu) + P_{l + 1}(\mu) Q_0(\mu)).
\end{equation}
Now the unknown values $\{ w_j \}_{j = 1}^l \backslash \{ w_j \}_{j = m - d + 1}^{m - 1}$ can be found analogously to Algorithm~\ref{alg:prime}. Thus, we arrive at the following algorithm for solution of Inverse Problem~\ref{ip:disc1}.

\begin{alg} \label{alg:part}
Suppose that $\gcd(m, l + 1) = d > 1$ and the numbers $\{ w_j \}_{j = m - d + 1}^{m - 1}$, $\{ \mu_n \}_{n = 1}^l \setminus \{ \mu_{kr} \}_{k = 1}^{d-1}$ are given. We have to find $\{ w_j \}_{j = 1}^l \backslash \{ w_j \}_{j = m - d + 1}^{m - 1}$.

\begin{enumerate}
    \item Add $\mu_{kr} = 2 \cos \frac{\pi k}{d}$, $k = \overline{1, d-1}$, $r = \frac{l + 1}{d}$, to the set of the given eigenvalues and construct the polynomial $D(\mu)$ by~\eqref{prod}.
    Determine $w_m$ by using~\eqref{Dwm}.
    \item Find $P_0(\mu)$ and $P_{l + 1}(\mu)$ by \eqref{Ppsi}.
    \item Find the polynomial $\vv_{m-d}(\mu)$ by \eqref{defvv} and its roots $\{ \nu_k \}_{k = 1}^{m - d}$.
    \item Using the given values $\{ w_j \}_{j = m - d + 1}^{m - 1}$, calculate $\{ Q_0^{\bullet}(\nu_k) \}_{k = 1}^{m - d}$ by \eqref{Q0bul}.
    \item Interpolate the polynomial $Q_0^{\bullet}(\mu) \in \mathscr P_{m - d - 1}$ by its values $\{ Q_0^{\bullet}(\nu_k) \}_{k = 1}^{m - d}$.
    \item Find $Q_0(\mu)$ from \eqref{Q0b}.
    \item Construct $Q_{l + 1}(\mu)$ by \eqref{findQl}.
    \item Implement steps~7-8 of Algorithm~\ref{alg:prime} to
    find the unknown values among $\{ w_j \}_{j = 1}^l$. 
\end{enumerate}
\end{alg}

Inverse Problem~\ref{ip:disc2} can be solved analogously. The constructive solutions of Inverse Problems~\ref{ip:disc1}-\ref{ip:disc2} imply the following uniqueness theorem.

\begin{thm} \label{thm:uniqd}
Suppose that $\gcd(m, l + 1) = d > 1$. Then a part of the spectrum of the problem~\eqref{eqv3}-\eqref{bc3} degenerates: $\mu_{kr} = 2 \cos \frac{\pi k}{d}$, $k = \overline{1, d-1}$, $r = \frac{l + 1}{d}$.

(i) If $w_j = \tilde w_j$ for $j = \overline{m - d + 1, m - 1}$ and $\mu_n = \tilde \mu_n$ for $n = \overline{1, l}$, then $w_j = \tilde w_j$, $j = \overline{1, l}$.

(ii) If $w_j = \tilde w_j$ for $j = \overline{m + 1, m + d}$ and $\mu_n = \tilde \mu_n$ for $n = \overline{1, l}$, then $w_j = \tilde w_j$, $j = \overline{1, l}$.
\end{thm}

Note that the condition $\mu_n = \tilde \mu_n$, $n = \overline{1, l}$, in (i) and (ii) can be replaced by $\mu_n = \tilde \mu_n$, $n \in \{ 1, \ldots, l \} \setminus \{ kr \}_{k = 1}^{d-1}$, since the degenerate eigenvalues automatically coincide.
Obviously, the spectrum degeneration described in Theorem~\ref{thm:uniqd} implies~\eqref{degh} for the eigenvalues of \eqref{eqv1}-\eqref{bc1}.

\begin{example} \label{ex:2}
Consider the special case $l = 2m-1$, which corresponds to $a = \frac{\pi}{2}$ in the continuous problem~\eqref{eqv2}-\eqref{bc2}, that is, the frozen argument is in the middle of the interval. Then, $d = m$ and formula~\eqref{char2} takes the form
\begin{equation} \label{relD2}
D(\mu) = \psi_m(\mu) (Q_0(\mu) + Q_{l + 1}(\mu)).
\end{equation}
Hence, $\mu_n = \frac{\pi n}{2m}$ for even $n$ and
\begin{equation} \label{Q0lprod}
Q_0(\mu) + Q_{l + 1}(\mu) = \prod_{\substack{n = 1 \\ n \: \text{is odd}}}^l (\mu - \mu_n).
\end{equation}
On the other hand, relations~\eqref{Qpsi} imply
\begin{equation} \label{Q0l}
Q_0(\mu) + Q_{l + 1}(\mu) = \psi_{m + 1}(\mu) - \psi_{m-1}(\mu) + w_m \psi_m(\mu) + \sum_{j = 1}^{m-1} (w_j + w_{l + 1 - j}) \psi_j(\mu).
\end{equation}
Thus, given the eigenvalues $\{ \mu_n \}$ with odd indices $n$, one can construct the polynomial $(Q_0(\mu) + Q_{l + 1}(\mu))$ by \eqref{Q0lprod} and find the coordinates $\{ w_j + w_{l + 1 - j} \}_{j  = 1}^{m-1}$, $w_m$ of the polynomial 
$$
(Q_0 + Q_{l + 1} - \psi_{m + 1} + \psi_{m-1}) \in \mathscr P_{m-1}
$$
with respect to the basis $\{ \psi_j \}_{j = 1}^m$. Obviously, it is impossible to uniquely determine the coefficients $\{ w_j \}_{j = 1}^l \setminus \{ w_m \}$. However, the uniqueness of the inverse problem solution holds in the symmetric case and if $\{ w_j \}_{j = 1}^{m-1}$ or $\{ w_j \}_{j = m + 1}^l$ are known a priori.
\end{example}

\begin{thm} \label{thm:uniq2}
Suppose that $l = 2m-1$ and one of the following conditions (i)-(iii) is fulfilled: 

(i) $w_j = w_{l + 1 - j}$, $\tilde w_j = \tilde w_{l + 1 - j}$, $j = \overline{1, m-1}$;

(ii) $w_j = \tilde w_j$, $j = \overline{1, m-1}$;

(iii) $w_j = \tilde w_j$, $j = \overline{m + 1, l}$.

Then, it follows from $\mu_n = \tilde \mu_n$, $n = \overline{1, l}$, that $w_j = \tilde w_j$, $j = \overline{1, l}$.
\end{thm}

\section{Eigenvalue asymptotics} \label{sec:asympt}

In this section, we study the connection between the eigenvalues of the continuous problem~\eqref{eqv2}-\eqref{bc2} and its finite-difference approximation~\eqref{eqv1}-\eqref{bc1}. For definiteness, we consider the case $a = \frac{\pi}{2}$ corresponding to $l = 2m-1$. The other cases can be studied similarly. We improve the asymptotics of the eigenvalues $\{ \la_n \}$ obtained in \cite{BBV19, HBY20} and
investigate the behavior of $\la_{n,l}$ as $l \to \iy$. 
Our methods differ from the ones used for the finite-difference approximation of the Sturm-Liouville equation \cite{PHA81}. Our approach is based on the expansion~\eqref{Q0l} by the Chebyshev polynomials. Three different techniques are developed for ``small'', ``middle'', and ``large'' indices $n$. As a result, the ``correction terms'' for approximation of $\la_{n,l}$ by $\la_n$ are obtained.

For the continuous problem \eqref{eqv2}-\eqref{bc2}, we briefly recall the results of~\cite{BBV19}. Equation~\eqref{eqv2} has the solutions
\begin{equation} \label{defCS}
    C(x, \la) := \cos (\rho(x - a)) + \int_a^x q(t) \frac{\sin (\rho(x - t))}{\rho} \, dt, \quad S(x, \la) := \frac{\sin (\rho(x - a))}{\rho},
\end{equation}
where $\rho = \sqrt{\la}$. Note that the definitions of $C(x, \la)$ and $S(x, \la)$ do not depend on the choice of the square root branch. The eigenvalues $\{ \la_n \}_{n = 1}^{\iy}$ of the problem~\eqref{eqv2}-\eqref{bc2} coincide with the zeros of the characteristic function
\begin{equation} \label{defDelta}
\Delta(\la) = \begin{vmatrix} C(0, \la) & S(0, \la) \\ C(\pi,\la) & S(\pi, \la) \end{vmatrix}.
\end{equation}
Substituting~\eqref{defCS} into~\eqref{defDelta}, we derive the relation
$$
\Delta(\la) = \frac{1}{\rho} \sin \frac{\rho \pi}{2} \left( 2 \cos \frac{\rho \pi}{2} + \int_0^{\pi/2} p(t) \frac{\sin (\rho t)}{\rho}\, dt \right),
$$
where $p(t) := q(t) + q(\pi - t)$. Obviously, the spectrum of \eqref{eqv2}-\eqref{bc2} consists of the degenerate eigenvalues $\la_{2n} = (2n)^2$, $n \in \mathbb N$, and of the non-degenerate eigenvalues $\la_n = \rho_n^2$ with odd $n$, where $\{ \rho_n \}$ are the zeros of the function
\begin{equation} \label{defR}
R(\rho) := 2 \cos \frac{\rho \pi}{2} + \int_0^{\pi/2} p(t) \frac{\sin (\rho t)}{\rho}\, dt
\end{equation}
such that $\arg \rho_n \in \left[-\frac{\pi}{2}, \frac{\pi}{2}\right)$. Using \eqref{defR}, we improve the asymptotics~\eqref{asymptla0}. 

\begin{lem} \label{lem:cont}
Suppose that $q \in L_2(0, \pi)$ and $a = \frac{\pi}{2}$. Then the eigenvalues $\{ \la_n \}$ of the boundary value problem~\eqref{eqv2}-\eqref{bc2} with odd $n$ have the asymptotics
\begin{equation} \label{asymptlac}
    \la_n = n^2 + \frac{2 \sin \frac{n\pi}{2}}{\pi} \int_0^{\pi/2} p(t) \sin (n t) \, dt + \frac{\varkappa_n}{n}.
\end{equation}
\end{lem}

Here and below, the notation $\{ \varkappa_n \}$ is used for various sequences from $l_2$.

\begin{proof}
Theorem~3.1 from~\cite{BBV19} implies that, for $q \in L_2(0, \pi)$ and odd $n$, $\la_n = \rho_n^2$, where
\begin{equation} \label{rhon}
\rho_n = n + \eps_n, \quad \eps_n = \frac{\varkappa_n}{n}.
\end{equation}
Substitute \eqref{rhon} into the equality $R(\rho_n) = 0$. Calculations show that
\begin{align*}
\cos \left( \frac{(n + \eps_n) \pi}{2} \right) & = -\sin \frac{n\pi}{2} \sin \frac{\eps_n \pi}{2} = - \frac{\eps_n \pi }{2} \sin\frac{n\pi}{2} + O(\eps_n^3), \\
\sin ((n + \eps_n) t) & = \sin (nt) (1 + O(\eps_n^2)) + \cos (nt) O(\eps_n) = \sin (nt) + O(\eps_n),
\end{align*}
uniformly with respect to $t \in [0, \pi/2]$. Hence
\begin{equation} \label{epsn}
\eps_n \pi \sin \frac{n \pi}{2} = \frac{1}{n}\int_0^{\pi/2} p(t) \sin (nt) \, dt + O(n^{-1} \eps_n).
\end{equation}
This together with~\eqref{rhon} imply
$$
\eps_n = \frac{\sin \frac{n\pi}{2}}{\pi n} \int_0^{\pi/2} p(t) \sin (nt) \, dt + \frac{\varkappa_n}{n^2}.
$$
Consequently, we arrive at relation~\eqref{asymptlac} for $\la_n = (n + \eps_n)^2$.
\end{proof}

Now return to the discrete inverse problem which was considered in Example~\ref{ex:2}. Since $l = 2m-1$, the eigenvalues $\la_{n,l}$ for even $n$ degenerate (see \eqref{degh}). Therefore, it is worth considering only $\la_{n,l}$ for odd $n$. Using the change of variables $\la = \frac{2 - \mu}{h^2}$, $\mu = 2 \cos \theta$, $w_j = h^2 q_j$ and relations~\eqref{sin}, \eqref{relD2}, \eqref{Q0l}, we conclude that $\la_{n,l} = \frac{2 - 2 \cos \theta_{n,l}}{h^2}$, where $\theta_{n,l}$ are the roots of the equation
\begin{equation} \label{cosd}
2 \cos (m \theta) = -h^2 q_m \frac{\sin (m \theta)}{\sin \theta} - h^2 \sum_{j = 1}^{m-1} (q_j + q_{l + 1 - j}) \frac{\sin (j\theta)}{\sin \theta}.
\end{equation}
For brevity, introduce the notations $p_j := q_j + q_{l + 1 - j}$ and
$$
\sum_{j = 0}^m{\vphantom{\sum}}' a_j = \frac{1}{2} a_0 + \sum_{j =1}^{m-1} a_j + \frac{1}{2} a_m.
$$
Obviously, $p_j = p(x_j)$, $p_m = 2 q_m$. Now equation~\eqref{cosd} can be rewritten in the form
\begin{equation} \label{eqcos}
2 \cos (m \theta) = -\frac{h^2}{\sin \theta} \sum_{j = 0}^m{\vphantom{\sum}}' p_j \sin (j \theta).
\end{equation}

Let us consider the behavior of the roots of equation~\eqref{eqcos} as $l = 2m-1$ tends to $\infty$ and obtain the asymptotics for the eigenvalues $\la_{n,l}$. We will use different techniques in the three different cases.

\textbf{``Small'' $n$.} Consider equation~\eqref{eqcos} in the circle $|\theta| \le C_0 h$, where $C_0 > 0$ is an arbitrary fixed constant. Note that $\sin (j \theta) = \sin \left( \frac{\theta}{h} x_j\right)$, where the sine argument is bounded. If $p \in C^2[0, \pi]$, the trapezoidal rule for integration implies
\begin{equation} \label{trap}
\int_0^{\pi/2} p(x) \sin\left( \frac{\theta}{h} x\right) \, dx =
h \sum_{j = 0}^m {\vphantom{\sum}}' p_j \sin \left( \frac{\theta}{h} x_j\right) + O(h^2).
\end{equation}
Using~\eqref{eqcos}, \eqref{trap}, the change of variables $\xi := \frac{\theta}{h}$, and taking $\frac{h}{\sin \theta} \sim \frac{1}{\xi}$ into account, we arrive at the equation
\begin{equation} \label{eqxi}
2 \cos \frac{\xi\pi}{2} + \frac{1}{\xi} \int_0^{\pi/2} p(x) \sin (\xi x) \, dx = O(h^2),
\end{equation}
where the $O$-estimate is uniform with respect to $\xi$ in the circle $|\xi| \le C_0$. Observe that the left-hand side of \eqref{eqxi} coincides with the function $R(\xi)$ defined by~\eqref{defR}. Recall that the zeros of this function are $\rho_n$ with odd indices $n$. Using Rouch\'e's Theorem, one can easily show that equation~\eqref{eqxi} has zeros $\xi_{n,l} = \rho_n + \eps_{n,l}$, where $\eps_{n,l} = o(1)$ for a fixed odd $n$ and $l \to \iy$. If $\rho_n$ is a simple zero of $R(\xi)$, we have $R'(\rho_n) \ne 0$, so
the Taylor expansion yields
$$
R(\rho_n + \eps_{n,l}) = R(\rho_n) + R'(\rho_n) \eps_{n, l} + O(\eps_{n,l}^2) = O(h^2) \quad \Rightarrow \quad \eps_{n,l} = O(h^2).
$$
Hence,
$$ 
\theta_{n,l} = h \xi_{n,l} = h \rho_n + O(h^3), \quad
\la_{n,l} = \frac{4 \sin^2 \frac{\rho_n h}{2}}{h^2} + O(h^2).
$$
Since $h \to 0$ and $\rho_n^2 = \la_n$, we arrive at the following theorem.

\begin{thm} \label{thm:small}
Suppose that $a = \frac{\pi}{2}$, $q \in C^2[0, \pi]$, and $\la_n \ne 0$ is a simple eigenvalue of \eqref{eqv2}-\eqref{bc2} with an odd $n$ (i.e. this eigenvalue does not belong to the degenerate set $\la_{2n} = (2n)^2$, $n \in \mathbb N$). Then the following asymptotic relation holds for $l = 2m-1 \to \iy$:
\begin{equation} \label{smla}
\la_{n,l} = \la_n + O(h^2), 
\end{equation}
where the $O$-estimate is uniform with respect to $n \le n_0$ for an arbitrary fixed integer $n_0$.
\end{thm}

\textbf{``Middle'' $n$.} Clearly, if all $p_j = 0$, equation \eqref{eqcos} has the roots 
\begin{equation} \label{theta0}
\theta_{n,l}^0 = \frac{\pi n}{2 m}, \quad n = \overline{1, 2m-1}, \quad n \:\: \text{is odd}. 
\end{equation}
For each fixed $l = 2m-1$, consider such indices $n$ that $\frac{n}{2m} \ge \al$, $\frac{2m-n}{2m} \ge \al$, $\al > 0$ is fixed. Consider the circles $|\theta - \theta_{n,l}^0| = r_{l,\de}$ with the radius $r_{l,\de} = \frac{\de}{l}$, $0 < \de < \frac{\pi}{2}$. One can easily check that
$$
|\sin (j \theta)| \le C, \quad \frac{1}{|\sin \theta|} \le C, \quad |\cos (m\theta)| \ge c_0,
$$
for $|\theta - \theta_{n,l}^0| = r_{l,\de}$, sufficiently large values of $l$, and a fixed $\de$. Here $C$ and $c_0$ are positive constants independent of $n$, $l$, $\theta$. Suppose that the function $p(x)$ is bounded. Then, the right-hand side of \eqref{eqcos} is $O(h)$ as $l \to \iy$. Using Rouch\'e's Theorem, we prove that equation \eqref{eqcos} has exactly one root $\theta_{n,l}$ in the circle $|\theta - \theta_{n,l}^0| < r_{l,\de}$ for sufficiently large $l$. Since $\de > 0$ can be arbitrarily small, we have 
\begin{equation} \label{asympttheta1}
\theta_{n,l} = \theta_{n,l}^0 + \eps_{n,l}, \quad \eps_{n,l} = o(h), \quad l \to \iy.
\end{equation}
Using~\eqref{eqcos}, \eqref{theta0}, and \eqref{asympttheta1}, we obtain
$$
2 \cos(m \theta_{n,l}) = - 2 \sin \left( \frac{n\pi}{2} \right) \sin (m \eps_{n,l}) = O(h) \quad \Rightarrow \quad \eps_{n,l} = O(h^2).
$$
Hence,
$$
\cos (m \theta_{n,l}) = -\sin \left(\frac{n\pi}{2}\right) m \eps_{n,l} + O(h^3), \quad
\sin (j \theta_{n,l}) = \sin(n x_j) + O(h).
$$
Substituting the latter relations into~\eqref{eqcos}, we derive
\begin{equation} \label{asympteps}
\eps_{n,l} = \frac{h^3}{\pi \sin(\theta_{n,l})} \sin \left( \frac{n\pi}{2} \right) \left( \sum_{j = 0}^m{\vphantom{\sum}}' p_j \sin (n x_j) \right) + O\left( \frac{h^3}{\sin (\theta_{n,l})}\right).
\end{equation}
Relations \eqref{asympttheta1} and \eqref{asympteps} imply the following asymptotics for the eigenvalues $\la_{n,l} = \frac{2 - 2 \cos \theta_{n,l}}{h^2}$:
\begin{gather} \label{asymptla1}
\la_{n,l} = \frac{4 \sin^2 \left( \frac{nh}{2}\right)}{h^2} + \frac{2}{\pi} \sin \left( \frac{ n \pi}{2}\right) \left( h \sum_{j = 0}^m{\vphantom{\sum}}' p_j \sin (n x_j) \right) + O(h), \quad l = 2m-1 \to \iy, \\ \nonumber \frac{n}{2m} \ge \al, \quad \frac{2m-n}{2m} \ge \al, \quad \al > 0.
\end{gather}

\textbf{``Large'' $n$}. Now consider the case $0 < \frac{2m-n}{2m} \le \al$, $\al > 0$. Let us study equation~\eqref{eqcos} for $|\theta - \theta_{n,l}^0| \le r_{l,\de}$, where $\theta_{n,l}^0$ and $r_{l,\de}$ are defined similarly to the case of ``middle'' $n$. Applying the change of variables $\phi := \pi - \theta$ in \eqref{eqcos}, we derive the relation
\begin{equation} \label{eqcos2}
2 (-1)^m \cos (m \phi) = -\frac{h^2}{\sin \phi} \sum_{j = 0}^m{\vphantom{\sum}}' p_j (-1)^j \sin(j\phi).
\end{equation}
Note that
\begin{equation} \label{difpj}
p_j \sin(j\phi) - p_{j + 1} \sin((j + 1)\phi) = (p_j - p_{j + 1}) \sin(j \phi) + p_{j + 1} (\sin(j\phi) - \sin((j + 1)\phi)).
\end{equation}
If $p \in C^1[0, \pi]$, then $p_j - p_{j + 1} = O(h)$. For $\phi$ corresponding to $|\theta - \theta_{n,l}^0| \le r_{l, \de}$, the following estimate holds:
$$
\frac{\sin (j \phi) - \sin((j + 1) \phi)}{\sin \phi} = O(1).
$$
Therefore, using \eqref{difpj}, we conclude that the right-hand side of \eqref{eqcos2} is $O(h)$. Following the proof for ``middle'' $n$, we arrive at the asymptotics~\eqref{asymptla1} for $\frac{2m-n}{2m} \le \al$. The results for ``middle'' and ``large'' $n$ are combined together in the following theorem.

\begin{thm} \label{thm:ml}
Suppose that $q \in C^1[0, \pi]$, $a = \frac{\pi}{2}$, and $\al > 0$ is fixed. Then, the asymptotic relation \eqref{asymptla1} holds for $l = 2m-1 \to \iy$ and odd $n$ uniformly with respect to $n$ satisfying $\frac{n}{2m} \ge \al$.
\end{thm}

Comparing \eqref{asymptlac} and \eqref{asymptla1}, we observe that \begin{equation} \label{approx}
\la_n - n^2 \approx \la_{n,l} - \frac{4 \sin^2 \left( \frac{nh}{2}\right)}{h^2}
\end{equation}
for ``middle'' and ``large'' $n$ in view of the trapezoidal rule
$$
    \int_0^{\pi/2} p(x) \sin (nx) \, dx \approx h \sum_{j = 0}^m{\vphantom{\sum}}' p_j \sin (n x_j).
$$
In the case of ``small'' $n$, we have $n^2 \sim \frac{4 \sin^2 \left( \frac{nh}{2}\right)}{h^2}$ as $h \to 0$, so \eqref{approx} also holds by virtue of Theorem~\ref{thm:small}. Therefore, given the eigenvalues $\{ \la_n \}$ of the continuous problem~\eqref{eqv2}-\eqref{bc2}, one can use the ``correction terms''  $\frac{4 \sin^2 \left( \frac{nh}{2}\right)}{h^2} - n^2$ to approximate the eigenvalues $\{ \la_{n,l} \}$ of the discrete problem~\eqref{eqv1}-\eqref{bc1}. This idea will be used for numerical solution of the inverse problem in Section~\ref{sec:num}.

Let us improve the results of Lemma~\ref{lem:cont} and Theorem~\ref{thm:ml} for $q \in W_1^2[0, \pi]$, $q(0) + q(\pi) = 0$,
$$
W_1^2[0, \pi] = \{ f \colon f, f' \in AC[0, \pi], \: f'' \in L(0, \pi) \}.
$$

\begin{lem} \label{lem:C2}
Suppose that $p \in W_1^2[0, \pi/2]$, $p(0) = 0$, $l = 2m-1$, and $\al > 0$ is fixed. Then
\begin{align} \label{contC2}
&    \int_0^{\pi/2} p(x) \sin (nx) \, dx = O(n^{-2}), \\ \label{discC2}
&    \sum_{j = 0}^m{\vphantom{\sum}}' p_j \sin (nx_j) = O(h), 
\end{align}
for odd $n$. The $O$-estimate in \eqref{discC2} is uniform for $\frac{n}{2m} \ge \al$.
\end{lem}

\begin{proof}
Integration by parts yields
$$
\int_0^{\pi/2} p(x) \sin (nx) \, dx = -\frac{p(x) \cos (nx)}{n}  \bigg|_0^{\pi/2} + \frac{p'(x) \sin(nx)}{n^2} \bigg|_0^{\pi/2} - \frac{1}{n^2} \int_0^{\pi/2} p''(x) \sin(nx) \, dx
$$
Under the assumptions of the lemma, this implies \eqref{contC2}.

Let us prove \eqref{discC2}. The function $p \in W_1^2[0, \pi/2]$ can be represented as the Fourier series
\begin{equation} \label{Fourier}
p(x) = \sum_{\substack{r = 1 \\ r \: \text{is odd}}}^{\iy} p^{(r)} \sin (r x),
\end{equation}
where the notation $p^{(r)}$ is used for the Fourier coefficients:
\begin{equation} \label{defpr}
p^{(r)} = \frac{4}{\pi} \int_0^{\pi/2} p(x) \sin (rx) \, dx.
\end{equation}
By virtue of~\eqref{contC2}, $p^{(r)} = O(r^{-2})$, so the series \eqref{Fourier} converges absolutely and uniformly with respect to $x \in [0, \pi/2]$.
Substituting~\eqref{Fourier} into the left-hand side of~\eqref{discC2} and changing the integration order, we derive
\begin{align*} 
\sum_{j = 0}^m {\vphantom{\sum}}' p_j \sin (nx_j) & = \sum_{j = 0}^m {\vphantom{\sum}}' \sum_{\substack{r = 1 \\ r \: \text{is odd}}}^{\iy} p^{(r)} \sin (r x_j) \sin(n x_j) \\ & = \frac{1}{2} \sum_{\substack{r = 1 \\ r \: \text{is odd}}}^{\iy} p^{(r)} \sum_{j = 0}^m {\vphantom{\sum}}' (\cos((r - n)x_j) - \cos((r + n) x_j)).
\end{align*}
Calculations show that
\begin{equation*}
\sum_{j = 0}^m {\vphantom{\sum}}' \cos (kx_j) = \mbox{Re} \, \sum_{j = 0}^m {\vphantom{\sum}}' \exp\left( \frac{ikj\pi}{2m} \right) = \begin{cases} 
m, \quad k = 4 ms, \: s \in \mathbb Z, \\
0, \quad \text{otherwise}.
\end{cases}
\end{equation*}
Hence
$$
\sum_{j = 0}^m {\vphantom{\sum}}' p_j \sin (nx_j) = \frac{m}{2} \left( \sum_{\substack{r = 1 \\ r - n = 4ms}}^{\iy} p^{(r)} - \sum_{\substack{r = 1 \\ r - n = 4ms}}^{\iy} p^{(r)}\right),
$$
where $r$ and $n$ are odd, and the both series absolutely converge, since $p^{(r)} = O(r^{-2})$. Using \eqref{defpr} and taking the inequality $\al \le \frac{n}{2m} < 1$ into account, we derive
\begin{align*}
\sum_{j = 0}^m {\vphantom{\sum}}' p_j \sin (nx_j) = & \frac{m}{2} \sum_{s \colon n + 4ms \ge 1} \frac{4}{\pi}\int_0^{\pi/2} p(x) \sin ((n + 4ms) x) \, dx \\ & - \frac{m}{2} \sum_{s \colon -n + 4ms \ge 1} \frac{4}{\pi}\int_0^{\pi/2} p(x) \sin ((-n + 4ms) x) \,dx \\
= & \frac{2 m}{\pi} \sum_{s = -\iy}^{\iy} \int_0^{\pi/2} p(x) \sin((n + 4ms)x) \, dx = m \cdot O\left( \sum_{s = -\iy}^{\iy} \frac{1}{(n + 4ms)^2}\right).
\end{align*}
Since $\frac{1}{n} \le \frac{h}{\pi \al}$, we have
$$
\sum_{s = -\iy}^{\iy} \frac{1}{(n + 4ms)^2} \le \frac{2}{n^2} + 2\sum_{s = 1}^{\iy} \frac{1}{(4ms)^2} \le Ch^2.
$$
Combining the latter results together, we arrive at~\eqref{discC2}.
\end{proof}

\begin{thm} \label{thm:C2}
Suppose that $q \in W_1^2[0, \pi]$, $q(0) + q(\pi) = 0$, $a = \frac{\pi}{2}$, $l = 2m-1$, and $\al > 0$ is fixed. Then, for odd $n$,
\begin{align} \label{asymptc2}
\la_n & = n^2 + O(n^{-2}), \\ \label{asymptd2} \la_{n,l} & = \frac{4 \sin^2\left( \frac{nh}{2}\right)}{h^2} + O(h^2), \quad \text{uniformly for} \:\: \frac{n}{2m} \ge \al.
\end{align}
\end{thm}

\begin{proof}
The conditions of Lemma~\ref{lem:C2} readily follow from the conditions of Theorem~\ref{thm:C2}.
Using \eqref{epsn} and \eqref{contC2}, we conclude that $\eps_n = O(n^{-3})$. This implies \eqref{asymptc2}.

Using~\eqref{discC2}, we show that $\eps_{n,l} = O(h^3)$ in \eqref{asympteps}. Consequently, one can easily prove that the remainder term in \eqref{asympteps} can be estimated as $O\left( \frac{h^4}{\sin(\theta_{n,l})}\right)$, and so~\eqref{asymptla1} takes the form~\eqref{asymptd2}.
\end{proof}

\begin{cor}
Under the assumptions of Theorem~\ref{thm:C2}, we have
$$
\la_n - n^2 = \la_{n,l} - \frac{4 \sin^2\left( \frac{nh}{2}\right)}{h^2} + O(h^2) \quad \text{uniformly for} \:\: \frac{n}{2m} \ge \al.
$$
\end{cor}

Thus, under the assumptions of Theorem~\ref{thm:C2}, we obtain a better approximation rate in \eqref{approx}, comparing the with the one given by Lemma~\ref{lem:cont} and Theorem~\ref{thm:ml}. 

\section{Numerical simulation} \label{sec:num}

In this section, we develop a numerical method for recovering the potential $q(x)$ of the problem~\eqref{eqv2}-\eqref{bc2} by a finite set of the eigenvalues. For this purpose, we approximate the eigenvalues of the discrete problem~\eqref{eqv1}-\eqref{bc1}, using the ``correction terms'' obtained in Section~\ref{sec:asympt}, and then solve the discrete inverse problem by the method of Section~\ref{sec:disc}. First, the numerical algorithm is described. Second, its effectiveness is illustrated by numerical examples.

For definiteness, we study the case $a = \frac{\pi}{2}$ and $l = 2m-1$, considered in Example~\ref{ex:2} and Section~\ref{sec:asympt}. For the uniqueness of the reconstruction, we assume that the potential is symmetric: $q(x) = q(\pi - x)$, $x \in [0, \pi]$. 
Since $\la_n$ with even $n$ degenerate, we fix $l = 2m-1$, $m \in \mathbb N$ and define the set $\mathcal N := \{ n = 2k-1\colon k = \overline{1, m} \}$. We seek an approximate solution of the following problem.

\begin{ip} \label{ip:num}
Given the eigenvalues $\{ \la_n \}_{n \in \mathcal N}$, find the potential $q(x)$.
\end{ip}

Relying on relation~\eqref{approx} and the solution described in Example~\ref{ex:2}, we develop the following numerical algorithm for solving Inverse Problem~\ref{ip:num}.

\begin{alg} \label{alg:num}
Let the eigenvalues $\{ \la_n \}_{n \in \mathcal N}$ be given. We have to find $\tilde q_j \approx q(x_j)$, $j = \overline{1,l}$.

\begin{enumerate}
    \item Find the approximations $\tilde \la_{n,l}$ for the eigenvalues $\la_{n,l}$ of the discrete problem and the numbers $\mu_n$ for $n \in \mathcal N$ as follows:
\begin{equation} \label{lat}
\tilde \la_{n,l} := \la_n - n^2 + \frac{4 \sin^2\left( \frac{nh}{2}\right)}{h^2}, \quad \mu_n := 2 - h^2 \tilde \la_{n,l}, \quad n \in \mathcal N.
\end{equation}
\item Construct the polynomial $Z \in \mathcal P_{m-1}$ by the formula
$$
Z(\mu) = \prod_{n \in \mathcal N} (\mu - \mu_n) - \psi_{m + 1}(\mu) + \psi_{m-1}(\mu).
$$
\item Find the coefficients $\{ z_j \}_{j = 1}^m$ of $Z$ with respect to the basis $\{ \psi_j \}_{j = 1}^m$.
\item Put $\tilde q_j := \frac{1}{2} h^2 z_j$, $j = \overline{1, m-1}$, and $\tilde q_m := h^2 z_m$.
\end{enumerate}
\end{alg}

Consider the examples of numerical reconstruction by Algorithm~\ref{alg:num} for the potentials $q(x) = x(\pi - x)$, $q(x) = \pi/2 - |\pi/2 - x|$, and $q(x) = 1$. The corresponding functions $R(\rho)$ defined by~\eqref{defR} have the form
\begin{align*}
    q(x) = x(\pi - x) : & \quad R(\rho) = 2 \cos \frac{\rho \pi}{2} - \frac{\pi^2}{2\rho^2} \cos \frac{\rho \pi}{2} + 
    \frac{4}{\rho^4} \left( 1 - \cos \frac{\rho \pi}{2} \right), \\
    q(x) = \frac{\pi}{2} - \left| \frac{\pi}{2} - x\right| : & \quad R(\rho) = 2 \cos \frac{\rho \pi}{2} - \frac{\pi}{\rho^2} \cos \frac{\rho \pi}{2} + \frac{2}{\rho^3} \sin \frac{\rho \pi}{2},  \\
    q(x) = 1 : & \quad R(\rho) = 2 \cos \frac{\rho \pi}{2} + \frac{2}{\rho^2} \left( 1 - \cos \frac{\rho \pi}{2} \right).
\end{align*}
We use the dichotomy to find the zeros of these functions numerically and so obtain the eigenvalues $\la_n$.

The numerical results for $m = 5$ are provided below.
In the tables, $\tilde \la_{n,l}$ are defined by~\eqref{lat}, $\de_{n,l} := \la_{n,l} - \tilde \la_{n,l}$, $\tilde q_j$ are the values constructed at step~4 of Algorithm~\ref{alg:num}, $\tilde q_j \approx q(x_j)$, $\de_j := q(x_j) - \tilde q_j$. The potentials $q(x)$ are symmetric, so they are considered only on the half-interval $[0, \pi/2]$.

\begin{gather*}
q(x) = x(\pi - x): \\
\begin{array}{|c|c|c|c|c|c|}
\hline
& \la_1 & \la_3 & \la_5 & \la_7 & \la_9 \\
\hline
\la_n & 3.5895 & 8.8607 &  25.0226 &  48.9922 &  81.0036 \\
\hline
\la_{n,l} & 3.5867  & 8.2083 & 20.2868 & 32.1684 & 39.5384 \\
\hline
\tilde \la_{n,l} & 3.5813  &  8.2139 &  20.2869 &  32.1674 &  39.5403 \\
\hline
\de_{n,l} &  0.0054 & -0.0056 & -0.0001 &  0.0009 & -0.0019 \\
\hline
& q_1 & q_2 & q_3 & q_4 & q_5 \\
\hline
\tilde q_j & 0.8857 & 1.5719 & 2.0705 &  2.3665 & 2.4686 \\ 
\hline
\de_j & 0.0025 & 0.0072 & 0.0021 & 0.0022 & -0.0012 \\
\hline
\end{array} 
\end{gather*}
\begin{gather*}
q(x) = \frac{\pi}{2} - \left| \frac{\pi}{2} - x \right|: \\
\begin{array}{|c|c|c|c|c|c|}
\hline
& \la_1 & \la_3 & \la_5 & \la_7 & \la_9 \\
\hline
\la_n & 2.2432 & 9.1668 & 25.0542 & 49.0268 & 81.0160 \\
\hline
\la_{n,l} & 2.2375 & 8.5351 & 20.3321 & 32.2168 & 39.5705
 \\
\hline
\tilde \la_{n,l} & 2.2350 & 8.5200 & 20.3184 & 32.2021 & 39.5527
 \\
\hline
\de_{n,l} & 0.0024 & 0.0152 & 0.0137 & 0.0148 &   0.0178
 \\
\hline
& q_1 & q_2 & q_3 & q_4 & q_5 \\
\hline
\tilde q_j & 0.3159 & 0.6330 & 0.9441 & 1.2665 & 1.5070 
 \\ 
\hline
\de_j & -0.0017 & -0.0047 & -0.0016 & -0.0099 & 0.0638
 \\
\hline
\end{array} 
\end{gather*}
\newpage
\begin{gather*}
q(x) = 1: \\
\begin{array}{|c|c|c|c|c|c|}
\hline
& \la_1 & \la_3 & \la_5 & \la_7 & \la_9 \\
\hline
\la_n & 2.3477 &  8.4962 &  25.2631 &  48.8138 &  81.1431 \\
\hline
\la_{n,l} & 2.3303 & 7.8801 & 20.4725 & 32.0695 & 39.5689
 \\
\hline
\tilde \la_{n,l} &  2.3395 & 7.8494 & 20.5274 & 31.9890 & 39.6797
 \\
\hline
\de_{n,l} &  -0.0092 & 0.0307 & -0.0549 & 0.0804 & -0.1108 \\
\hline
& q_1 & q_2 & q_3 & q_4 & q_5 \\
\hline
\tilde q_j & 1.1752 & 0.8892 & 1.0747 & 0.9328 & 1.0639 \\ 
\hline
\de_j & -0.1752 & 0.1108 & -0.0747 & 0.0672 & -0.0639 \\
\hline
\end{array}
\end{gather*}

In Figure~\ref{fig:1}, the plots of $q(x) = x(\pi - x)$ and $q(x) = \pi/2 - |\pi/2 - x|$ on $[0, \pi/2]$ are provided. The stars indicate the values $\tilde q_j$, $j = \overline{1, m}$, recovered by Algorithm~\ref{alg:num} with $m = 5$.

\begin{figure}[h!]
\begin{center}
\includegraphics[scale = 0.18]{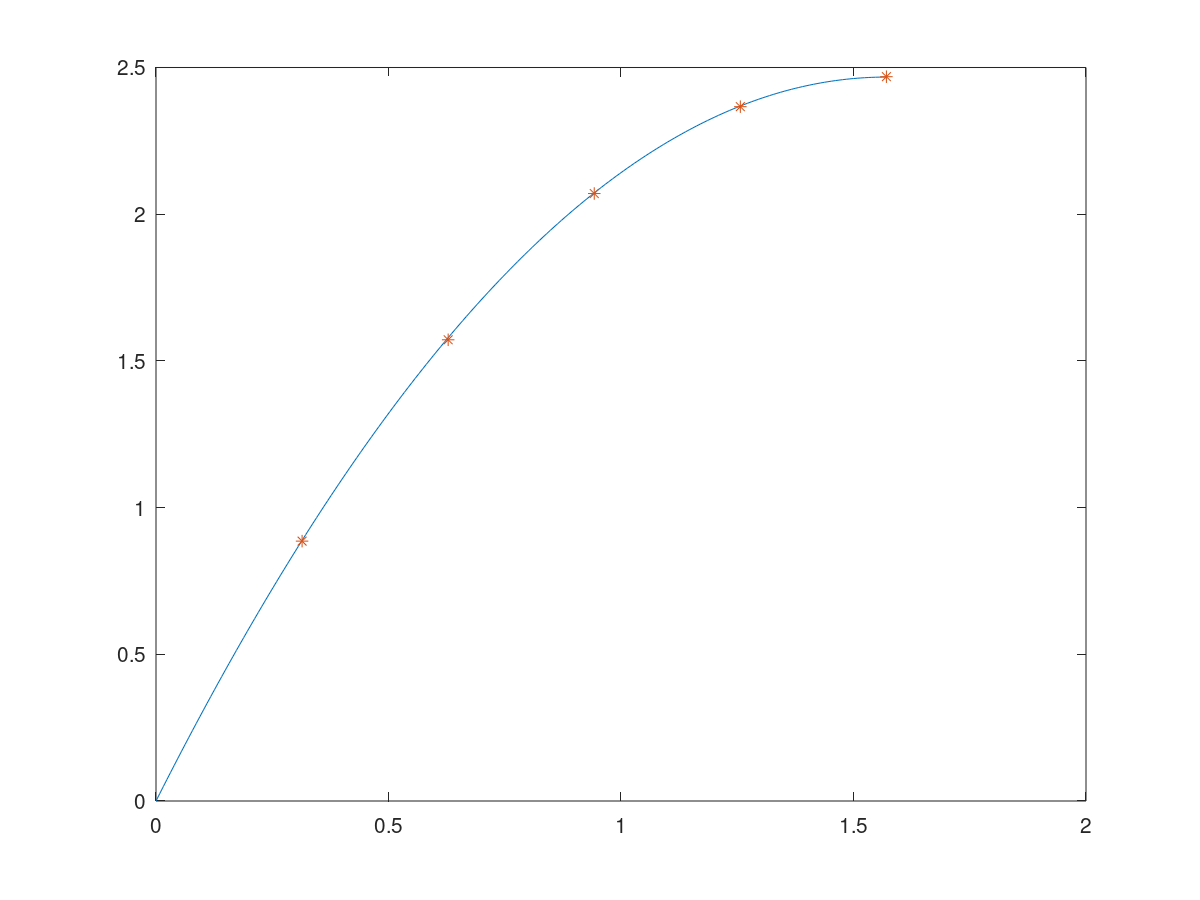}
\includegraphics[scale = 0.18]{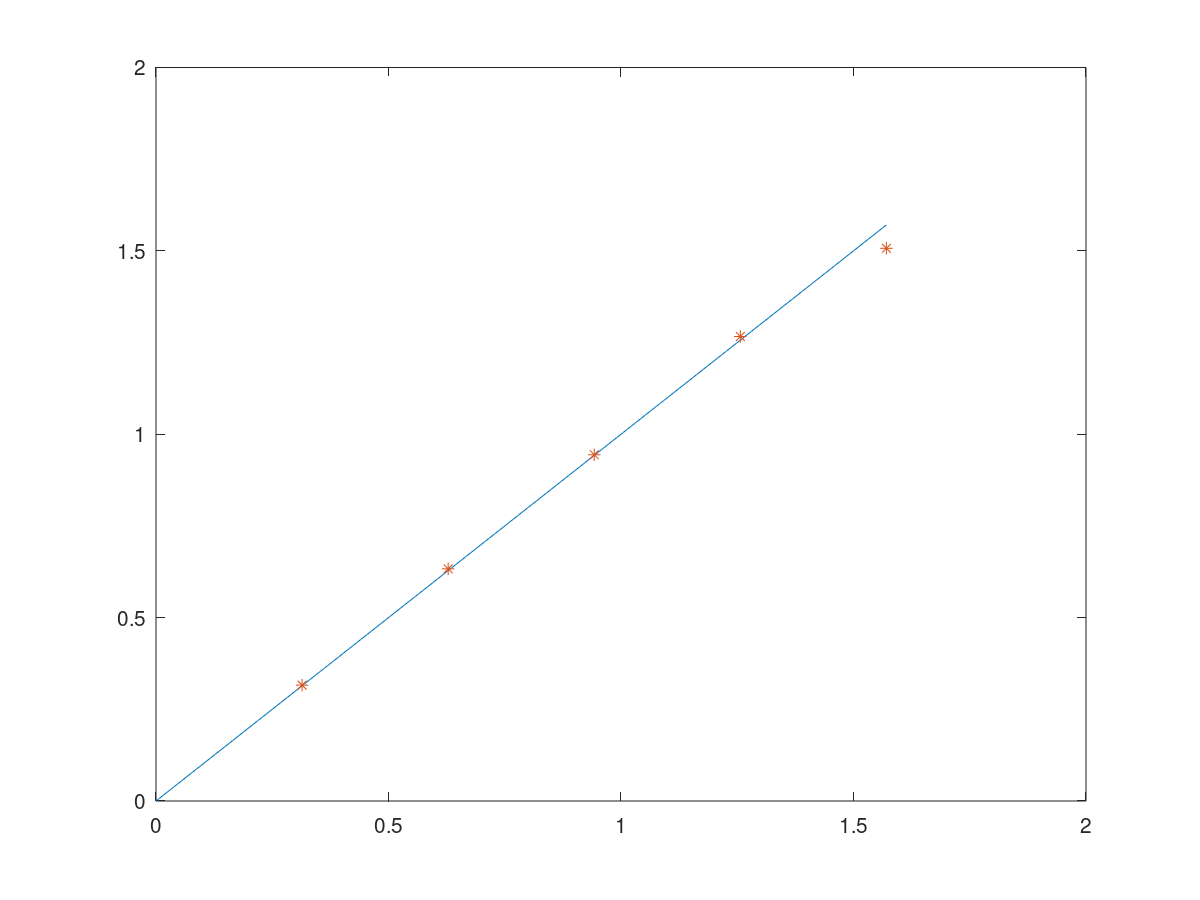}
\end{center}
\caption{Plots of $q(x) = x(\pi - x)$ and $q(x) = \pi/2 - |\pi/2 - x|$}
\label{fig:1}
\end{figure}

The results for $m = 10$ and $m = 20$ are provided in the data set \cite{DS}. These data show that the approximation error is decreasing while $m$ is growing. Anyway, in the first two examples, the potential is already recovered with significant accuracy for $m = 5$. The potential $q(x) = 1$ does not fulfill the condition $q(0) + q(\pi) = 0$, so the accuracy decreases near $x = 0$. 

\section{Generalizations and applications} \label{sec:gen}

Let us discuss the possible ways of the further development of our approach.

\begin{enumerate}
    \item For the functional-differential equation~\eqref{eqv2}, inverse spectral problems have been studied for various types of boundary conditions (the Neumann BC, periodic BC, etc.; see \cite{BBV19, BV19, BK20, BH21}). Those types of boundary conditions imply a richer structure of degenerate and non-degenerate cases than the Dirichlet ones. The methods of this paper can be applied to finite-difference approximations for other types of boundary conditions.
    \item Observe that one can choose $l$ and $m$ so that $x_m = a$ if and only if $\frac{a}{\pi} \in \mathbb Q$. If $\frac{a}{\pi} \not\in \mathbb Q$, then $\frac{a}{\pi}$ should be approximated by irreducible fractions $\frac{m}{l + 1}$. Studying the approximation of the continuous problem~\eqref{eqv2}-\eqref{bc2} by the discrete one \eqref{eqv1}-\eqref{bc1} in this case is a challenging task.
    \item Note that the matrix of the system \eqref{eqv1}-\eqref{bc1} contains the only column induced by the potential $q$. Therefore, the Laplace expansion of the characteristic determinant along this column
    leads to the expansion by the Chebyshev polynomials. The similar effect takes place for the three-diagonal matrix with an additional row of unknown coefficients. The latter case occurs for finite-difference approximations of differential operators with integral boundary conditions. For example, consider the eigenvalue problem
    $$
        -y''(x) = \la y(x), \quad x \in (0, \pi), \qquad \int_0^{\pi} v(x) y(x) \, dx = 0, \quad y(\pi) = 0.
    $$
    Its finite-difference approximation has the form
    $$
        -\frac{y_{j + 1} - 2y_j + y_{j-1}}{h^2} = \la y_j, \quad j = \overline{1, l}, \qquad \sum_{k = 0}^l v_k y_k = 0, \quad y_{l + 1} = 0,
    $$
    where the coefficients $[v_k]_{k = 0}^l$ are determined by~$v(x)$ according to some quadrature rule. Clearly, the matrix of the latter system has the row of the coefficients $[v_k]_{k = 0}^l$, so the transposed matrix has the form studied in this paper. This observation is related with the fact that the functional-differential operators with frozen argument are adjoint to differential operators with integral boundary conditions (see, e.g., \cite{Lom14, Pol21}). Consequently, the methods of our paper can be applied to the latter operator class.
    \item The ideas of our approach can be developed for investigation of inverse spectral problems for other classes of nonlocal operators, in particular, for integro-differential operators \cite{But07} and for differential operators with constant delay \cite{PVV19}.
\end{enumerate}

{\bf Data availability.} There is the data set \cite{DS} associated with this paper.

\noindent Natalia Pavlovna Bondarenko \\
1. Department of Applied Mathematics and Physics, Samara National Research University, \\
Moskovskoye Shosse 34, Samara 443086, Russia, \\
2. Department of Mechanics and Mathematics, Saratov State University, \\
Astrakhanskaya 83, Saratov 410012, Russia, \\
e-mail: {\it bondarenkonp@info.sgu.ru}

\end{document}